\documentclass[psamsfonts]{amsart}
\usepackage[all]{xy}

\theoremstyle{plain}
\newtheorem{thm}{Theorem}[section]

\newtheorem{lemma}[thm]{Lemma}
\newtheorem{cor}[thm]{Corollary}

\newtheorem{condition}[thm]{Condition}
\newtheorem{construction}[thm]{Construction}

\newtheorem{example}[thm]{Example}

\theoremstyle{remark}
\newtheorem{rem}[thm]{Remark}

\begin{document}
\date{\today}
\subjclass[2010]{Primary 11R29, 11Y16; Secondary 11R45}

\title [Explicit Approximations to Class Field Towers]{ Explicit Approximations to Class Field Towers}

\author[Bleher]{F. M. Bleher}
\address{F. M. Bleher, Dept. of Mathematics\\Univ. of Iowa \\ Iowa City, IA 52246, USA}
\email{frauke-bleher@uiowa.edu}
\thanks{The first author was partially supported by NSF grant DMS 1801328, Simons Foundation grant 960170 and
NSF FRG grant DMS-2411703.}

\author[Chinburg]{T. Chinburg}
\address{T. Chinburg, Dept. of Mathematics\\ Univ. of Pennsylvania \\ Philadelphia, PA 19104, USA}
\email{ted@math.upenn.edu}
\thanks{The second author is the corresponding author; he was partially supported by  NSF SaTC grant No. 1701785,
Simons Foundation grant No. MP-TSM-00002279 and  NSF FRG grant DMS-2411702.}

\maketitle

\begin{abstract} We answer a question of Peikert and Rosen by giving for each $\epsilon > 0$ an efficient  construction of infinite families of number fields $N$ such that  the root discriminant $D_N^{1/[N:\mathbb{Q}]}$ is bounded above by a constant  times $[N:\mathbb{Q}]^\epsilon$.  
\end{abstract}

\section{Introduction}
\label{s:intro}

In this paper we will say that a family of number fields can be efficiently constructed if there is an algorithm and a polynomial $f(x) \in \mathbb{R}[x]$ such that for each field $F$ in the family, the algorithm requires time bounded by $f(\log [F:\mathbb{Q}])$ for producing a set of polynomials whose roots generate $F$.   Our main result is:

\begin{thm}
\label{thm:overtheorem}
For each $\epsilon > 0$, there is an efficient construction of an infinite family of number fields $E$ of increasing degree $n$ such that $D_{E}^{1/n} = O(n^{\epsilon}).$ 
 \end{thm}
 
By a result of Golod and Shafarevitch  \cite{Roquette},  there are infinite class field towers, and the fields in such a tower have constant root discriminant. However, we do not know of an efficient construction of the fields in a class field tower.  

Theorem \ref{thm:overtheorem} answers affirmatively a question of Peikert and Rosen  \cite[\S 1.2]{PeikertRosen}.  They asked whether there is an efficient construction of a family of number fields $E$ of increasing degree $n$ for which  $D_{E}^{1/n} = o(n \log \log (n)/\log(n)).$
 
 We will call the families of  $E$ we construct in Theorem \ref{thm:overtheorem} explicit approximations to class field towers, since they arise as large degree subfields of the second level in the two-elementary class field towers of a family of quadratic fields.  We now describe one instance of the construction.

 \begin{construction}
 \label{con:doit}
 \hspace*{1cm}
\begin{enumerate}
\item[1.]  For a large $X > 0$ let $p_1,\ldots,p_t$ be the increasing sequence of all rational primes $p_i$ such that $p_i \equiv 1$ mod $8$ and $p_i \le X$.
\item[2.]  For each ordered pair $(p_i,p_j)$ with $i \ne j$ calculate the quadratic residue symbol $\left ( {p_i} \atop {p_j}\right ) $.   We continue with this pair if and only if  $\left ( {p_i} \atop {p_j}\right ) = 1$.
\item[3.]  Using continued fractions, calculate fundamental units $e_i$, $e_j$ and $e_{ij}$ of the real quadratic fields $\mathbb{Q}(\sqrt{p_i})$, $\mathbb{Q}(\sqrt{p_j})$ and $\mathbb{Q}(\sqrt{p_i \cdot p_j})$. By multiplying $e_{ij}$ by $-1$ if necessary we can assume that at least one conjugate of $e_{ij}$ is positive. Write $e_i = a_i + b_i \frac{1+ \sqrt{p_i}}{2}$ for some explicit $a_i, b_i \in \mathbb{Z}$.
\item[4.]  Since $\left ( {p_i} \atop {p_j}\right ) = 1$, we can find an element $\beta_i \in (\mathbb{Z}/p_j)^*$ such that $\beta_i^2 \equiv p_i$ in $\mathbb{Z}/p_j$.  Using $e_i = a_i + b_i \frac{1+ \sqrt{p_i}}{2}$ from step 3, calculate  $e'_i = a_i + b_i \frac{1 + \beta_i}{2} \in \mathbb{Z}/p_j$.  If the quadratic symbol $\left ({e'_i} \atop {p_j} \right )$ equals $1$, discard the pair 
$(p_i,p_j)$.  Otherwise define $(p_i,p_j)$ to be in the set $S$ of ``useful ordered pairs" of primes.  
 \item[5.]   Suppose $(p_i,p_j) \in S$. If both conjugates of $e_{ij}$ are positive then $u = \sqrt{e_{ij}}$ lies in the ring of integers $O_L$ of the field $L = \mathbb{Q}(\sqrt{p_i},\sqrt{p_j})$.  Otherwise, there will be a unique choice of $s \in \{\pm 1\}$ such that 
 $u = \sqrt{ s e_i e_j e_{ij}}$ lies in $O_L$.  There will be a unique quadruple $(a, b, c, d) \in \{0,1\}^4$ such that 
$w_{ij} = (-1)^a e_i^b  e_j^c u^d \ne 1$ and  $\psi(w_{ij})$ 
 lies in $\{1,5\}$ mod $8$ for all four of the surjective algebra homomorphisms $\psi:O_L \to \mathbb{Z}/8$ that result from the fact that $2$ splits in $L$. \end{enumerate}
\end{construction}

We now state the following precise form of Theorem \ref{thm:overtheorem}.

\begin{thm}
\label{thm:main}  Let $N$ be the Galois CM extension of $\mathbb{Q}$ generated by all of the numbers $\sqrt{p_i}$, $\sqrt{p_j}$ and $\sqrt{w_{ij}}$ for all  $(p_i,p_j) \in S$, where $S$ is the set of useful prime pairs defined in step 4 of Construction \ref{con:doit}.   One has  
\begin{equation} 
\label{eq:bigineq}
\log{[N:\mathbb{Q}]} \ge c \frac{X^2}{(\log(X))^2} \quad \mathrm{and} 
\quad \log{ (D_{N}^{1/[N:\mathbb{Q}]})} \le \frac{1}{2} X \log(X)
\end{equation}
for some effective constant $c > 0$.  In consequence, 
$$D_{N}^{1/[N:\mathbb{Q}]} = O([N:\mathbb{Q}]^\epsilon)$$
for all $\epsilon > 0$, where the constant is effective and depends only on $\epsilon$.  The family of fields $N$ that can be constructed in this way as $X$ varies can be constructed efficiently.
\end{thm} 

One consequence of the above theorem is that fields of small root discriminant can be produced  by adjoining to $\mathbb{Q}$ fourth roots of products of $\pm 1$ and fundamental units of real quadratic fields whose discriminants involve only one or two primes.  There is a significant advantage to being able to construct $N$ by adjoining to $L$ the square roots of units.  For example, we will show elsewhere that this implies the ring of integers $O_N$ of $N$ is a free module over the ring of integers of $L$, with a basis that can be described explicitly using the $\sqrt{w_{i,j}}$ appearing in Theorem \ref{thm:main}. This is relevant to the use of $O_N$ in cryptography.

To make the proof unconditional we rely on a strong conditional Chebotarev density theorem of Pierce, Turnage-Butterbaugh and Wood in \cite{PTW}. We use work of Kowalski and Michel in \cite{KM}  to show that for most of the  fields we must consider, the hypotheses of this conditional theorem are met.  It is fortunate that the zeta functions that arise are products of $L$-functions associated to either Dirichlet characters or odd two-dimensional Galois representations.  By work of Deligne and Serre, the Ramanujan-Petersson conjecture is known for these $L$-functions, and one has strong convexity bounds for them, due to their connection to automorphic forms.  

Higher dimensional automorphic representations have also been used to construct particular number fields of high degree and small root discriminant. A particular example is due to Demb\'el\'e \cite{Dem} and Serre \cite{Serre}. They use a Hilbert modular form of level $1$ and parallel weight $2$ over the maximal totally real subfield of the cyclotomic field $\mathbb{Q}(\zeta_{32})$ to construct a non-solvable Galois extension $K$ of $\mathbb{Q}$ that is unramified outside of $2$, having degree $2^{19}\cdot 3^2 \cdot  5^2 \cdot 17^{2} \cdot 257^2$ and root discriminant $\le 55.394388\ldots$.
However, their construction does not provide an explicit basis for $K$ over $\mathbb{Q}$.

There are many variations on Construction \ref{con:doit} that lead to sharper results with more complicated statements. For example, a variation on this construction produces a field $N'$ with 
\begin{equation}
\label{eq:Nprimediscriminant}
[N': \mathbb{Q}] = 256  \quad \mathrm{and}\quad   
D_{N'}^{1/[N':\mathbb{Q}]} = 2^{7/4}\cdot \sqrt{3 \cdot 7\cdot 13} \approx 55.5756...
\end{equation} 
By comparison, when $\zeta$ is a root of unity of order $512 = 2^9$ one has
$$D_{\mathbb{Q}(\zeta)}^{1/[\mathbb{Q}(\zeta):\mathbb{Q}]} = [\mathbb{Q}(\zeta): \mathbb{Q}] = 256.$$
Thus from the point of view of root discriminants, the construction of $N'$  improves on the two-power cyclotomic case by more than a factor of $4$.
Theorem \ref{thm:main} says this improvement only becomes greater in higher degrees.

To describe $N'$ explicitly, 
define
 \begin{equation}
 \label{eq:nudef}
 \nu_{i,j} = -13 +(-1)^i 5 \sqrt{7} - (-1)^j 8 \sqrt{13} + (-1)^{i+j} 3 \sqrt{7} \cdot \sqrt{13}
 \end{equation}
 for $i, j \in \{1,2\}$.   Then $N'$ is obtained by adjoining to $\mathbb{Q}$ all square roots of elements of  
 \begin{equation}
 \label{eq:Cdef}
 C = \left\{7,13,-1, \frac{3 + \sqrt{13}}{2}\right\} \cup \{\nu_{i,j}\}_{1 \le i, j \le 2}.
 \end{equation}
 Note that we have used the primes $7$ and $13$ that are not congruent to $1$ mod $8$, and $N'$ will not be a CM field.  
  In fact, $N'$ is the ray class field of $\mathbb{Q}(\sqrt{7},\sqrt{13})$ of conductor $6$ times the product of the infinite places of this field.

We now give an outline of this paper.

In \S \ref{s:algebra} we prove various algebraic and group theoretic results that reduce the proof of Theorem \ref{thm:main} to the problem of showing $\# S$ is at least as large as a positive constant times $(X/\log(X))^2$ for all sufficiently large $X$.  This statement is reduced to showing that for a particular family of degree $16$ fields parameterized by the primes $p_i \equiv 1$ mod $8$  with $p_i \le X$, one has for most elements in this family a strong Chebotarev density theorem. 
 In \S \ref{s:chebo} we recall the Chebotarev result from \cite{PTW} and the $L$-function results from \cite{KM} that are needed to   finish the proof of Theorem \ref{thm:main} in \S \ref{s:finish}.  In \S\ref{s:redei} we discuss a connection of R\'edei symbols and steps 1 through 4 of Construction \ref{con:doit}. In \S \ref{s:example} we discuss the example $N'$ above.

\section{Proof of Theorem \ref{thm:main}: Algebraic part}
\label{s:algebra}

Throughout this section we will assume the notations of steps 1 through 5 of Construction \ref{con:doit}.

\begin{lemma}
\label{lem:algpart1}  Suppose $(p_i,p_j) \in S$.  The class number of $L = \mathbb{Q}(\sqrt{p_i},\sqrt{p_j})$ is odd, and $(p_j,p_i) \in S$.  The unit group $O_L^*$ has generators $-1$, $e_i$, $e_j$ and $u$.    There is a unique
CM dihedral extension $F$ of $\mathbb{Q}$ that is quadratic over $L$ and unramified over all finite places of $L$.  One has $F = L(\sqrt{w_{ij}})$ when $w_{ij}$ is the unique product satisfying the conditions in step 5 of Construction \ref{con:doit}. We have $w_{ij} = w_{ji}$. 
\end{lemma}

\begin{proof}
 By genus theory applied to $\mathbb{Q}(\sqrt{p_i})/\mathbb{Q}$ and 
$\mathbb{Q}(\sqrt{p_j})/\mathbb{Q}$ the class numbers $h_{\mathbb{Q}(\sqrt{p_i})}$ and $h_{\mathbb{Q}(\sqrt{p_j})}$ are odd.  We know that $p_j$ splits in $\mathbb{Q}(\sqrt{p_i})$ by step 2, and a fundamental unit $e_i$ of $\mathbb{Q}(\sqrt{p_i})$ is not a square at a place over $p_j$ in $\mathbb{Q}(\sqrt{p_i})$ by step 4.  Hence by genus theory for $L/\mathbb{Q}(\sqrt{p_i})$ we conclude that $h_L$ is odd.  

We now want to show that $(p_j,p_i) \in S$.  The condition in step 2 holds for $(p_j,p_i)$ by quadratic reciprocity.  So we have to show the condition in step 4 for $(p_j,p_i)$. Suppose the condition in step 4 fails for $(p_j,p_i)$.  Then 
$e_j$ is a square modulo one of the primes over $p_i$ in $\mathbb{Q}(\sqrt{p_j})$.  However, the conjugate of $e_j$ over $\mathbb{Q}$ is $\pm e_j^{-1}$ and $-1$ is a square mod $p_i$.  It follows the $e_j$ is a square modulo both primes over $p_i$ in $\mathbb{Q}(\sqrt{p_j})$ if the condition  in step 4 fails for $(p_j,p_i)$.  However, 
genus theory would now imply that $L$ has even class number, a contradiction. So the condition in step 4  holds for
$(p_j,p_i)$, and so $(p_j,p_i) \in S$.

Since $L/\mathbb{Q}(\sqrt{p_i p_j})$ is an unramified quadratic extension and $h_L$ is odd, we see that $2$ divides $h_{\mathbb{Q}(\sqrt{p_i p_j})}$ but $4$ does not.  It is a classical result of Herglotz (see \cite{He} and \cite[Thm. 3.1]{Sime}) that
\begin{equation}
\label{eq:biquadclass}
h_L = \frac{1}{4}[O_L^*:O_{\mathbb{Q}(\sqrt{p_i})}^*\cdot O_{\mathbb{Q}(\sqrt{p_j})}^* 
\cdot O_{\mathbb{Q}(\sqrt{p_i p_j})}^* ] \cdot 
h_{\mathbb{Q}(\sqrt{p_i})} \cdot 
h_{\mathbb{Q}(\sqrt{p_j})} \cdot
h_{\mathbb{Q}(\sqrt{p_i p_j})}
\end{equation} 
Let $\{H_k\}_{k = 1}^3$ be the set of the three order two subgroups of $\mathrm{Gal}(L/\mathbb{Q}) \cong (\mathbb{Z}/2) \times (\mathbb{Z}/2)$.  The identity 
$$\sum_{k = 1}^3 (\sum_{\sigma \in H_k} \sigma ) - \sum_{\tau \in \mathrm{Gal}(L/\mathbb{Q})} \tau = 2$$
in the integral group ring of $\mathrm{Gal}(L/\mathbb{Q})$ shows that $(O_L^*)^2$ is contained in $J = O_{\mathbb{Q}(\sqrt{p_i})}^*\cdot O_{\mathbb{Q}(\sqrt{p_j})}^* 
\cdot O_{\mathbb{Q}(\sqrt{p_i p_j})}^* $.  Combining this with (\ref{eq:biquadclass}) and the above facts about class numbers shows $O_L^*/J$ has order $2$.
 Thus if $e_i$, $e_j$ and $e_{ij}$ are fundamental units of $\mathbb{Q}(\sqrt{p_i})$, $\mathbb{Q}(\sqrt{p_j})$ and $\mathbb{Q}(\sqrt{p_i p_j})$, there will be a unique quadruple $(c_0, c_i, c_j, c_{ij}) \in \{0,1\}^4$ that is not $(0,0,0,0)$ such that  $(-1)^{c_0} e_i^{c_i} e_j^{c_j} e_{ij}^{c_{ij}} = u^2$ for some $u \in O_L^*$.  By genus theory, $e_i$ and $e_j$ have norm $-1$ to $\mathbb{Q}$.  By considering the signs
of $(-1)^{c_0} e_i^{c_i} e_j^{c_j} e_{ij}^{c_{ij}} $ at infinity, we see that the vector $(c_0,c_i,c_j,c_{ij})$ is determined uniquely by the condition that
 $(-1)^{c_0} e_i^{c_i} e_j^{c_j} e_{ij}^{c_{ij}} $ is totally positive. In step 3 of Construction \ref{con:doit} we normalized $e_{ij}$ to have at least one positive embedding. This leads to being able to take $u^2 = e_{ij}$ if $e_{ij}$ is totally positive and $u^2 = s   e_i e_j e_{ij}$ otherwise for a unique choice of $s \in \{\pm 1\}$, as in step 5 of Construction \ref{con:doit}.  Furthermore $O_L^*$ is generated by $-1, e_i, e_j$ and $u$.
 
 The prime $p_j$ splits as $\mathcal{P} \cdot \mathcal{Q}$ in $\mathbb{Q}(\sqrt{p_i})$. Since $h_{\mathbb{Q}(\sqrt{p_i})}$ is odd and $e_i$ has norm $-1$, the maximal elementary abelian two-power quotient of the narrow ray class group of $\mathbb{Q}(\sqrt{p_i})$ modulo  $\mathcal{P} \cdot \mathcal{Q}$ is a Klein four group.  This quotient corresponds to a Klein four extension $F$ of $\mathbb{Q}(\sqrt{p_i})$ which is Galois over $\mathbb{Q}$ and is a quadratic extension of $L$ that is unramified over all finite primes of $L$.  Since $L$ has odd class number, this forces $F$ to be  the unique totally complex dihedral extension of $\mathbb{Q}$ that contains $L$ and is unramified over all the finite places of $L$.  We showed that the group of signs at infinity in $L$ generated by units has order at least $8$.  This and the fact that $L$ has odd class number shows that the two-part of the narrow classgroup of $L$ has order $2$. Thus $F$ is  the unique quadratic extension of $L$ that is unramified over all finite places of $L$.  We arranged that the prime $2$ splits completely in $L$, so Kummer theory and the fact that $h_L$ is odd now show $F = L(\sqrt{w_{ij}})$ for some $w_{ij}$ as in step 5 of Construction \ref{con:doit}.  Conversely, if $w_{ij}$ satisfies the conditions in step 5, $L(\sqrt{w_{ij}})$ is a quadratic extension of $L$ that is unramified over all finite places of $L$, so  we must have $F =L(\sqrt{w_{ij}})$.  This proves that $w_{ij}$ is uniquely determined by the conditions in step 5. For this reason, $w_{ij} = w_{ji}$. 
 \end{proof}

\begin{lemma}
\label{lem:under}  Let $N$ be the field defined in Theorem \ref{thm:main}, namely the extension of $\mathbb{Q}$ generated by all of the numbers $\sqrt{p_i}$, $\sqrt{p_j}$ and $\sqrt{w_{ij}}$ for all  $(p_i,p_j) \in S$.  Define $U$  to be the set of primes $p_i$ for which  $(p_i,p_j) \in S$ for some $p_j$, and let $\ell = \# U$. Then $N$ is a Galois CM extension of $\mathbb{Q}$ and $\mathrm{Gal}(N/\mathbb{Q})$ is a central extension of an elementary abelian two-group by an elementary abelian two-group.  We have  
\begin{equation}  \label{eq:lowershort}
 \mathrm{ord}_2([N:\mathbb{Q}]) =  \frac{\#S}{2} + \ell.
 \end{equation}
 The field $N$ is an elementary abelian extension of $F_1 = \mathbb{Q}(\{\sqrt{p_i}:p_i \in U\})$ that is unramified over all finite places of $F_1$.  
 \end{lemma}
 
 \begin{proof} 
One has  $\mathrm{Gal}(F_1/\mathbb{Q}) = (\mathbb{Z}/2)^\ell$ and $F_1$ is the maximal elementary abelian two-extension of $\mathbb{Q}$ that is unramified outside of $U $.  Clearly $F_1$ contains every biquadratic extension $L = \mathbb{Q}(\sqrt{p_i},\sqrt{p_j})$ associated to a pair $(p_i,p_j) \in S$. Recall $L(\sqrt{w_{ij}})$ is unramified over all finite places of $L$.
By definition, $N$ is the compositum over $\mathbb{Q}$ of all such $L(\sqrt{w_{ij}})$.  We conclude that $N$ is unramified over all finite places of $F_1$, and $\mathrm{Gal}(N/F_1)$ is an elementary abelian two-group that is central in
the group $\mathrm{Gal}(N/\mathbb{Q})$.  Furthermore, the elementary abelian two-group $\mathrm{Gal}(F_1/\mathbb{Q})$ is the maximal abelian quotient of $\mathrm{Gal}(N/\mathbb{Q})$. Hence the inertia groups in $\mathrm{Gal}(N/\mathbb{Q})$ of every place of $N$ over $p_i \in U$ have order $2$.  Finally $N$ is a CM field because complex conjugation is non-trivial and lies in $\mathrm{Gal}(N/F_1)$.  This implies the first statement of the lemma.

 Let $\mathcal{G}$ be the free pro-$2$ group on generators $x_1,\ldots,x_\ell$.  The first two terms in the descending central two-series of $\mathcal{G}$ are defined by 
 $$\mathcal{G}_1 = [\mathcal{G},\mathcal{G}] \cdot \mathcal{G}^2\quad \mathrm{and}\quad \mathcal{G}_2 = [\mathcal{G},\mathcal{G}_1] \cdot \mathcal{G}_1^2.$$
 Thus  $\mathcal{G}/\mathcal{G}_1$ is the maximal elementary abelian two-quotient of $\mathcal{G}$, and $\mathcal{G}/\mathcal{G}_2$ is the maximal quotient of $\mathcal{G}$ that is a central elementary abelian two-extension of an elementary abelian two-quotient of $\mathcal{G}$.  
 
 We map $\mathcal{G}$ to $\mathrm{Gal}(N/\mathbb{Q})$ by sending $x_i$ to a generator $x'_i$ of an inertia group of a place over $p_i$.    The images of the $x'_i$ in the maximal elementary abelian quotient $\mathrm{Gal}(F_1/\mathbb{Q})$ of $\mathrm{Gal}(N/\mathbb{Q})$ are a minimal set of generators for $\mathrm{Gal}(F_1/\mathbb{Q})$.  This forces the $x'_i$ to generate the two-group $\mathrm{Gal}(N/\mathbb{Q})$.  We
 arrive at a surjection
 \begin{equation}
 \label{eq:lambda}
 \lambda: \mathcal{G}/\mathcal{G}_2 \to \mathrm{Gal}(N/\mathbb{Q})
 \end{equation}
  which gives
 isomorphisms
 $$\mathcal{G}/\mathcal{G}_1 =   \frac{\mathrm{Gal}(N/\mathbb{Q})}{[\mathrm{Gal}(N/\mathbb{Q}),\mathrm{Gal}(N/\mathbb{Q})] \cdot \mathrm{Gal}(N/\mathbb{Q})^2} = \mathrm{Gal}(F_1/\mathbb{Q})$$
 on maximal elementary abelian two-quotients.  Accordingly, the kernel
 of $\lambda$ is a subgroup $\mathcal{R}$ of the $\mathbb{Z}/2$-vector space
 $$V = \mathcal{G}_1/\mathcal{G}_2.$$
 Here $V$ has a basis given by
 \begin{equation}
 \label{eq:Basis}
 B =  \{x_i^2 : 1 \le i \le \ell\}  \cup \{[x_i,x_j]: 1 \le i < j \le \ell \}.
 \end{equation}
 
 We now claim that if $(p_i,p_j) \in S$ and $i < j$, then no relation $r \in \mathcal{R}$ can involve a non-zero coefficient of $[x_i,x_j]$ when $r$ is written as a linear combination of elements of the basis $B$ with coefficients in $\mathbb{Z}/2$.
 Suppose to the contrary that such a relation exists.  Consider the 
 kernel of the surjection 
 $$\pi: \mathcal{G}/\mathcal{G}_2 \to \mathrm{Gal}(L(\sqrt{w_{ij}})/\mathbb{Q})$$
 when $L = \mathbb{Q}(\sqrt{p_i},\sqrt{p_j})$ for some $(p_i,p_j) \in S$.  Suppose $k \not \in \{i,j\}$. Since 
 $L(\sqrt{w_{ij}})/\mathbb{Q}$ is unramified at all finite places of $\mathbb{Q}$ outside of $\{p_i,p_j\}$, and $x_k$  goes to a generator of an inertia group over $p_k$, we see $x_k \in \mathrm{ker}(\pi)$.  We have $[x_i,x_k] = (x_i x_k x_i^{-1}) x_k^{-1}$
 and both $x_i x_k x_i^{-1}$ and $x_k^{-1}$ lie in inertia groups of primes over $p_k$.  So $[x_i,x_k] \in \mathrm{ker}(\pi)$ and similarly $[x_j,x_k] \in \mathrm{ker}(\pi)$.  Finally, $x_i^2$ and $x_j^2$ are in $\mathrm{ker}(\pi)$ since the inertia groups of $p_i$ and $p_j$ in $\mathrm{Gal}(L(w_{ij})/\mathbb{Q})$ have order $2$.  Thus $\mathrm{ker}(\pi)$ contains all elements of the basis $B$ in  (\ref{eq:Basis}) except possibly for $[x_i,x_j]$.  However, if a relation $r$ involves $[x_i,x_j]$ non-trivially, then
 since $r \in \mathrm{ker}(\pi)$ we would have $[x_i,x_j] \in \mathrm{ker}(\pi)$.  Hence $\mathcal{G}_1/\mathcal{G}_2 \subset \mathrm{ker}(\pi)$. Then $\pi$ factors  through $\mathcal{G}/\mathcal{G}_1 = \mathrm{Gal}(F_1/\mathbb{Q})$.  This is impossible because $\pi$ is surjective, $\mathrm{Gal}(F_1/\mathbb{Q})$ is an elementary abelian two-group and $\mathrm{Gal}(L(\sqrt{w_{ij}})/\mathbb{Q})$ is dihedral of order $8$.  
 (One can also prove this fact using more sophisticated arguments involving cup products;  see \cite{McS}.)
 
 This proves that if 
 $$B' = \{x_i^2: 1 \le i \le \ell\} \cup \{[x_k,x_z]: (p_k,p_z) \not \in S \quad \mathrm{and} \quad 1 \le k < z \le \ell\} \subset B$$
 then $\mathcal{R}$ is contained in the subgroup  
$V'$  of $\mathcal{G}_1/\mathcal{G}_2$ generated by $B'$.   Here $\lambda$ maps $V/\mathcal{R}$ injectively into $\mathrm{Gal}(N/F_1)$, and $V/\mathcal{R}$ surjects onto $V/V'$.  Therefore 
 $$ \# S/2 = \# B - \# B'  = \mathrm{dim}_{\mathbb{Z}/2} V/ V'  \le \mathrm{dim}_{\mathbb{Z}/2}V/\mathcal{R}   \le \mathrm{ord}_2(\# \mathrm{Gal}(N/F_1))$$ 
where the first equality follows from the fact that $(p_k,p_z) \in S$ if and only if $(p_z,p_k) \in S$.
 We end up with the bound
 \begin{equation}
 \label{eq:lowerb}
\mathrm{ord}_2([N:F_1]) \ge \#S/2.
 \end{equation}
 However, $N$ is obtained from $F_1$ by adjoining $\#S / 2$ square roots of elements of $F_1$, so (\ref{eq:lowerb}) must
 be an equality.  Since $ \mathrm{ord}_2([F_1:\mathbb{Q}]) = \ell$, this gives the equality in (\ref{eq:lowershort}) and completes the proof.
 \end{proof}
 
 \begin{cor}
 \label{cor:niceending}
 To prove Theorem \ref{thm:main}, it will suffice to show that there is a constant $c' > 0$ such that 
 \begin{equation}
 \label{eq:Slower}
 \#S /2 \ge c' \frac{X^2}{(\log (X))^2}.
 \end{equation}
 \end{cor}
 
 \begin{proof} The field $N$ is a Galois CM extension of $\mathbb{Q}$ by Lemma \ref{lem:under}, and (\ref{eq:Slower}) implies the first inequality in 
 (\ref{eq:bigineq}) with $c = c' \log(2) $.  For the second inequality in (\ref{eq:bigineq}), we know by Lemma 
\ref{lem:under} that $N$ is unramified over all finite places of $F_1 = \mathbb{Q}(\{\sqrt{p_i}:p_i \in U\})$.  Therefore
 $$D_N^{1/[N:\mathbb{Q}]} = D_{F_1}^{1/[F_1:\mathbb{Q}]} = \prod_{p_i \in U} p_i^{1/2} \le (X!)^{1/2}.$$
 This gives
 $$\log{ (D_{N}^{1/[N:\mathbb{Q}]})} \le \frac{1}{2} \log (X!) \le \frac{1}{2} X \log(X)$$
 as claimed.
 
 It remains to show that the field $N$ can be efficiently constructed.  Since $\log [N:\mathbb{Q}]$ is bounded below by a positive constant times $(X/\log X)^2$, it will suffice to show that one can write down a set of equations in time bounded by a polynomial in $X$ whose roots generate $N$.   By \cite[Thm. 5.5]{Lenstra} one can find the fundamental units of  $\mathbb{Q}(\sqrt{p_i})$, $\mathbb{Q}(\sqrt{p_j})$ and $\mathbb{Q}(\sqrt{p_i p_j})$ in time bounded by a polynomial in $X$.  Calculating quadratic residues mod $p_i$ is polynomial time in $X$, as is finding square roots of numbers mod $p_i$ that are known to be squares mod $p_i$.  Finally, Hensel's Lemma gives an algorithm that runs in time bounded by a polynomial in $X$ for calculating explicitly all surjective algebra homomorphisms from the integers of $\mathbb{Q}(\sqrt{p_i},\sqrt{p_j})$
 to $\mathbb{Z}/8$ when $p_i \ne p_j$ are primes congruent to $1$ mod $8$ that are bounded by $X$.
 Since the number of primes $p_i \le X$ is bounded by $X$, this shows $N$ can be efficiently constructed.
 \end{proof}
 
 In the next section we will recall some explicit Chebotarev theorems that are sufficient to prove a lower bound for $\# S$ of the required kind.  To apply these theorems we need some further algebraic results.

 \begin{lemma}
 \label{lem:Frobenius}  Let $p_i$ be a prime with $p_i \equiv 1 $ mod $8$ and $p_i \le X$.  Let $e_i$ be a fundamental unit of $\mathbb{Q}(\sqrt{p_i})$.  Define $Y_1 = \mathbb{Q}(\sqrt{-1},\sqrt{p_i}, \sqrt{e_i})$ and  $Y = \mathbb{Q}(\zeta_8,\sqrt{p_i}, \sqrt{e_i})$ when $\zeta_8$ is a primitive $8^{th}$ root of unity.  Then $Y_1$ is a  dihedral Galois extension of $\mathbb{Q}$ of degree $8$.  
 The field 
 $Y$ is the compositum of the disjoint fields $\mathbb{Q}(\sqrt{2})$ and $Y_1$.  Thus 
 \begin{equation}
 \label{eq:Yisom}
 \mathrm{Gal}(Y/\mathbb{Q}) = \mathrm{Gal}(\mathbb{Q}(\sqrt{2})/\mathbb{Q}) \times \mathrm{Gal}(Y_1/\mathbb{Q}).
 \end{equation}
 The non-trivial one dimensional complex characters of $\mathrm{Gal}(Y/\mathbb{Q})$ are the seven order two  characters factoring through $\mathrm{Gal}(\mathbb{Q}(\zeta_8,\sqrt{p_i})/\mathbb{Q}) \cong (\mathbb{Z}/2)^3$.  The conductors of these characters are $4$, $8$, $8$, $p_i$, $4p_i$, $8p_i$ and $8p_i$.  There are two irreducible odd dihedral characters $\rho$ and $\rho\cdot \chi_2$ of $\mathrm{Gal}(Y/\mathbb{Q})$, where $\rho$ factors through $\mathrm{Gal}(Y_1/\mathbb{Q})$ and $\chi_2$ is the non-trivial order two character factoring through $\mathrm{Gal}(\mathbb{Q}(\sqrt{2})/\mathbb{Q})$.    The conductors of  $\rho$ and $\chi_2\cdot \rho$ divide  $64 p_i$.  The discriminant $d_Y$ of $Y$ divides $2^{16} p_i^4 \cdot (64 p_i)^4 = 2^{40} p_i^8$.  
 The group  $\mathrm{Gal}(Y/\mathbb{Q}(\zeta_8,\sqrt{p_i}))$ is a central order two subgroup of $\mathrm{Gal}(Y/\mathbb{Q})$.  
   An odd prime $p_j \le X$ different from $p_i$ has $(p_i,p_j) \in S$ if and only if the Frobenius conjugacy class of $p_j$ in $\mathrm{Gal}(Y/\mathbb{Q})$ equals the non-trivial element of 
 $\mathrm{Gal}(Y/\mathbb{Q}(\zeta_8,\sqrt{p_i}))$.  \end{lemma}
 
 \begin{proof}    Since $\mathrm{Norm}_{\mathbb{Q}(\sqrt{p_i})/\mathbb{Q}}(e_i) = -1$,  the conjugates of $\sqrt{e_i}$ lie in $\{\pm \sqrt{e_i},\pm \sqrt{-1}/\sqrt{e_i}\}$.   It follows that $Y_1$ is Galois over $\mathbb{Q}$.  The extensions $\mathbb{Q}(\sqrt{e_i},\sqrt{p_i})$ and $\mathbb{Q}(\sqrt{-1},\sqrt{p_i})$ are quadratic and disjoint over $\mathbb{Q}(\sqrt{p_i})$ since they have different ramification at the infinite places of $\mathbb{Q}(\sqrt{p_i})$.  So $\mathrm{Gal}(Y_1/\mathbb{Q})$ has order $8$ and is non-abelian because $\mathbb{Q}(\sqrt{e_i},\sqrt{p_i})$ has two real places and one complex place.  The group 
 $\mathrm{Gal}(Y_1/\mathbb{Q}(\sqrt{p_i}))$ is a Klein four group, so $\mathrm{Gal}(Y_1/\mathbb{Q})$ must be dihedral of order $8$.  
 
 The maximal subfield of $Y_1$ that is abelian over $\mathbb{Q}$ is $\mathbb{Q}(\sqrt{-1},\sqrt{p_i})$ and this does not contain the subfield $\mathbb{Q}(\sqrt{2})$ of $\mathbb{Q}(\zeta_8) \subset Y$. Hence $Y$ is the compositum of the disjoint fields $\mathbb{Q}(\sqrt{2})$ and $Y_1$, and (\ref{eq:Yisom}) holds.  Because $Y$ is of degree $2$ over $\mathbb{Q}(\zeta_8,\sqrt{p_i})$ and $\mathbb{Q}(\zeta_8,\sqrt{p_i})$ is Galois over $\mathbb{Q}$, the group $\mathrm{Gal}(Y/\mathbb{Q}(\zeta_8,\sqrt{p_i}))$ is central of order two in $\mathrm{Gal}(Y/\mathbb{Q})$.
 
 The one dimensional  complex characters of $\mathrm{Gal}(Y/\mathbb{Q})$ are inflated from characters of $$\mathrm{Gal}(\mathbb{Q}(\zeta_8,\sqrt{p_i})/\mathbb{Q})\cong (\mathbb{Z}/2)^3$$
 and these have the indicated conductors.  Let $\rho$ be the inflation to $\mathrm{Gal}(Y/\mathbb{Q}) $ of the two-dimen\-sional irreducible representation of the dihedral group  $\mathrm{Gal}(Y_1/\mathbb{Q})$.  The center of $\mathrm{Gal}(Y_1/\mathbb{Q})$ is $\mathrm{Gal}(Y_1/\mathbb{Q}(\sqrt{-1},\sqrt{p_i}))$ and does not contain a complex conjugation.  Hence $\rho$ has eigenvalues $1$ and $-1$ on every complex conjugation, so $\rho$ is odd.  From (\ref{eq:Yisom}) we see that the other two-dimensional irreducible representation of $\mathrm{Gal}(Y/\mathbb{Q})$ is $\chi_2\cdot \rho$.

 The inertia groups of primes over $p_i$ in $\mathrm{Gal}(Y/\mathbb{Q})$ have order $2$ since $Y/\mathbb{Q}(\sqrt{p_i})$ is unramified over $p_i$.  These groups are not contained in the center $\mathrm{Gal}(\mathbb{Q}(\sqrt{2})/\mathbb{Q}) \times \mathrm{Gal}(Y_1/\mathbb{Q}(\sqrt{-1},\sqrt{p_i})) = \mathrm{Gal}(Y/\mathbb{Q}(\sqrt{-1},\sqrt{p_i}))$
 of $\mathrm{Gal}(Y/\mathbb{Q})$.  Hence the restriction of $\rho$ and $\rho \cdot \chi_2$ to an inertia group over $p_i$ is the sum of a trivial character and a quadratic non-trivial character, so $p_i$ exactly divides the conductors of these characters.  The inertia groups of primes over $2$ are contained in the group $\mathrm{Gal}(Y/\mathbb{Q}(\sqrt{p_i}))$, which is elementary abelian of order $8$. Since $2$ splits in $\mathbb{Q}(\sqrt{p_i})$, we find that the restriction of $\rho$ and $\rho \cdot \chi_2$ to an inertia group over $2$ is the sum of two irreducible one-dimensional characters that each have conductor dividing $8$. This leads to $\rho$ and $\rho \cdot \chi_2$ having conductor dividing $8^2 p_i$. These results on conductors show $d_Y$ divides $2^{40} p_i^8$ from the conductor discriminant formula.

 The conditions that $p_j \equiv 1$ mod $8$ and $\left ( {p_i} \atop {p_j}\right ) = 1 $ are equivalent to $p_j$ splitting in $\mathbb{Q}(\zeta_8,\sqrt{p_i})$.
 The conditions in step 4 of Construction \ref{con:doit} are equivalent to the primes over $p_j$ in
 $\mathbb{Q}(\zeta_8,\sqrt{p_i})$ not splitting in the quadratic extension $Y = \mathbb{Q}(\zeta_8,\sqrt{p_i},\sqrt{e_i})$ of $\mathbb{Q}(\zeta_8,\sqrt{p_i})$, and the Lemma is now clear from this.
 \end{proof}
  
\section{Chebotarev results from \cite{PTW} and \cite{KM}.}
\label{s:chebo}

We begin by recalling the  following conditional Chebotarev theorem proved by Pierce, Turnage-Butterbaugh and Wood in \cite[Theorem 3.1]{PTW}.

\begin{thm}
\label{thm:uncond}  Let $k$ be a fixed number field.  Fix $A \ge 2$, $0 < \delta \le 1/(2A)$ and an integer $n \ge 1$.
Let $G$ be a fixed transitive subroup of $S_n$.  Then there exist constants $D_0 \ge 1$, $\kappa_1, \kappa_2, \kappa_3 > 0$ such that the following holds.  Let $L/k$ be a Galois extension for which $\mathrm{Gal}(L/k)$ is isomorphic to $G$. Fix one such isomorphism.  Suppose $D_L \ge D_0$.  Suppose that the Artin $L$-function $\zeta_L(s)/\zeta_k(s)$ is zero-free in the region
\begin{equation}
\label{eq:freeregion}
[1 - \delta, 1] \times [-(\log D_L)^{2/\delta},(\log D_L)^{2/\delta}]
\end{equation}
in the complex plane.  Let $\mathcal{C} \subset G$ be a conjugacy class, and for $2 \le x \in \mathbb{R}$ let
$\pi_{\mathcal{C}}(x,L/k)$ be the number of finite places $\mathcal{P}$ of $k$ for which $\mathrm{Norm}_{k/\mathbb{Q}}(\mathcal{P}) \le x$, $\mathcal{P}$ is unramified in $L$ and the Frobenius conjugacy class of $\mathcal{P}$ in $G$ is $\mathcal{C}$.  Then for all 
\begin{equation}
\label{eq:lowerx1}
x \ge \kappa_1 \mathrm{exp}\{\kappa_2 (\log \log (D_L^{\kappa_3}))^2\}
\end{equation}
one has
\begin{equation}
\label{eq:Chebbound}
|\pi_{\mathcal{C}}(x,L/k) -  \frac{|\mathcal{C}|}{|G|} \mathrm{Li}(x)| \le  \frac{|\mathcal{C}|}{|G|}  \frac{x}{(\log x)^A}
\end{equation}
where $\mathrm{Li}(x)$ is the usual logarithmic integral
\begin{equation}
\label{eq:logint}
\mathrm{Li}(x) = \int_2^x \frac{dt}{\log t}. 
\end{equation}
If $k = \mathbb{Q}$, one can weaken the condition (\ref{eq:lowerx1}) to 
\begin{equation}
\label{eq:lowerx2}
x \ge \kappa_1 \mathrm{exp}\{\kappa_2 (\log \log (D_L^{\kappa_3}))^{5/3} (\log \log \log (D_L^2))^{1/3}\}.
\end{equation}
\end{thm}

We will eventually need to know that most elements of a particular family of extensions $L/k$ satisfy the hypotheses of Theorem \ref{thm:uncond}.  For this we recall some results of Kowalski and  Michel in \cite{KM} as formulated in \cite[\S 5.1]{PTW}.

Suppose $m \ge 1$ is given.  Let $\rho$ be a cuspidal automorphic representation of $\mathrm{GL}(m)/\mathbb{Q}$. 
Suppose $\alpha \in [1/2,1]$ and $T \ge 0$.  Define the zero-counting function associated to the $L$-function $L(s,\rho)$ by
\begin{equation}
\label{eq:zerocounting}
N(\rho;\alpha,T) = \# \{s = \beta + i \gamma: \beta \ge \alpha, |\gamma| \le T, L(s,\rho) = 0\}
\end{equation}
where each zero is counted with multiplicity.

We will be interested in families of representations satisfying the following conditions:

\begin{condition}
\label{cond:listit}
For $X \ge 1$ let $S(X)$ be a finite (possibly empty) set of cuspidal automorphic representations $\rho$ of
$GL(m)/\mathbb{Q}$ such that the following properties hold for $(S(X))_{X \ge 1}$. 
\begin{enumerate}
\item[i.] Every $\rho \in S(X)$ satisfies the Ramanujan-Petersson conjecture at finite places.
\item[ii.]  There exists $A > 0$ and a constant $M_0$ such that for all $X \ge 1$ and all $\rho \in S(X)$, $\mathrm{Cond}(\rho) \le M_0 X^A$.
\item[iii.] There exists $d > 0$ and a constant $M_1$ such that for all $X \ge 1$,  $|S(X)| \le M_1 X^d$.
\item[iv.] For any $\epsilon > 0$ there is a constant $M_{2,\epsilon}$ such that all $\rho \in S(X)$ satisfy the convexity bound
$$|L(s,\rho)| \le M_{2,\epsilon}(\mathrm{Cond}(\rho)(|t| + 1)^m)^{(1 - \mathrm{Re}(s))/(2 + \epsilon)} \quad \mathrm{for}\quad 0 \le \mathrm{Re}(s) \le 1 \quad \mathrm{and} \quad t = \mathrm{Im}(s).$$
For any $\epsilon > 0$ there is a constant $M_{3,\epsilon}$ such that for all $\rho \not \cong \rho' \in S(X)$ we have  the convexity bound
$$|L(s,\rho\otimes \rho')| \le M_{3,\epsilon}(\mathrm{Cond}(\rho\otimes \rho')(|t| + 1)^{m^2})^{(1 - \mathrm{Re}(s))/(2 + \epsilon)} \quad \mathrm{for}\quad 0 \le \mathrm{Re}(s) \le 1 \quad \mathrm{and} \quad t = \mathrm{Im}(s).$$
\end{enumerate}
\end{condition}

\begin{example}
\label{ex:exone}
{\rm For each rational prime $p_i \equiv 1$ mod $8$ let $L_{p_i}$ be the degree $16$ extension $Y$ of $\mathbb{Q}$ associated to $p_i$ in Lemma \ref{lem:Frobenius}.  Let $S(X)$ be the union of the $7$ one-dimensional nontrivial characters of $\mathrm{Gal}(L_{p_i}/\mathbb{Q})$ as $p_i$ ranges over all primes $p_i \equiv 1 $ mod $8$ such that $p_i \le X$.  We can regard each such $\rho$ as a cuspidal automorphic representation of $\mathrm{GL}(m)$ when $m = 1$. The Ramanujan-Petersson conjecture holds for $\rho$ by \cite[p. 95]{IK}.   Since $\mathrm{Cond}(\rho)$ divides $8p_i$ we can let $M_0 = 8$ and $A $ be any constant such that $A \ge 1$ in part (ii) of Condition \ref{cond:listit}.  In part (iii) we can let $M_1 = 7$ and $d = 1$.  By \cite[Ex. 3, p. 100]{IK} there are constants $M_{2,\epsilon}$ and $M_{3,\epsilon}$ for which 
condition (iv) of Condition \ref{cond:listit} holds.}
\end{example}

\begin{example}
\label{ex:extwo}
{\rm With the notations of Example $\ref{ex:exone}$, there are two irreducible odd dihedral Galois representations $\rho$
of $\mathrm{Gal}(L_{p_i}/\mathbb{Q})$.  These can be considered to be  cuspidal automorphic representations of $\mathrm{GL}(m)$ when $m = 2$ that correspond to holomorphic automorphic forms of weight $1$.  By work of Deligne and of Deligne and Serre (see \cite[p. 131]{IK}) they satisfy the Ramanujan-Petersson conjecture.  Let $S(X)$ be the union of the two two-dimensional irreducible characters of $\mathrm{Gal}(L_{p_i}/\mathbb{Q})$ as $p_i$ ranges over all primes $p_i \equiv 1 $ mod $8$ such that $p_i \le X$.  Since the conductors of these characters for $L_{p_i}$ divide $64 p_i$ by Lemma \ref{lem:Frobenius}, we can let $M_0 = 64$ and $A$ be any constant with $A \ge 1$ in part (ii) of Condition \ref{cond:listit}.   In part (iii) we can let $M_1 = 2$ and $d = 1$.   By \cite[Ex. 3, p. 100]{IK} there are constants $M_{2,\epsilon}$ and $M_{3,\epsilon}$ for which 
condition (iv) of Condition \ref{cond:listit} holds.} 
\end{example}

We now have the following result of Kowalski and Michel (see \cite[Theorem 2]{KM} and \cite[Theorem E]{PTW}).

\begin{thm}
\label{thm:bigCheb}  Let $(S(X))_{X \ge 1}$ be a family of cuspidal automorphic representations of $\mathrm{GL}(m)/\mathbb{Q}$ satisfying Condition \ref{cond:listit}.  Suppose $3/4 \le \alpha \le 1$ and $T \ge 2$.  There exists a constant $c'_0 = c'_0(m,A,d)$, which can be taken to be
\begin{equation}
\label{eq:c0prime}
c'_0 = \frac{5mA}{2} + d,
\end{equation}
and a constant $B \ge 0$ depending only on the parameters $m, A, d, M_0, M_1, M_{2,\epsilon}, M_{3,\epsilon}$ for which the following is true.  For every choice of a constant $c_0 > c'_0$ there is a constant $M_{4,c_0}$ depending only on $c_0$ such that for all $X \ge 1$, one has
\begin{equation}
\label{eq:bigupper}
\sum_{\rho \in S(X)} N(\rho;\alpha, T) \le M_{4,c_0} T^B X^{c_0(1-\alpha)/(2 \alpha - 1)}.
\end{equation}
\end{thm}

\begin{cor}
\label{cor:finally}
Suppose $S(X)$ is one of the two families described in Examples \ref{ex:exone} and \ref{ex:extwo}.  Set $A = 2$ in Theorem \ref{thm:uncond}.  For $p_i \equiv 1$ mod $8$ and $p_i \le X$, the group $G = \mathrm{Gal}(L_{p_i}/\mathbb{Q}) \cong \mathbb{Z}/2 \times D_8$ is a transitive subgroup of the symmetric group $S_n$ on $n = 16$ letters by means of its left action on itself.  Suppose $0 < \delta \le 1/(2A) = 1/4$ and that $\delta$ is sufficiently close to $0$.  Then there is a constant $0 \le \tau < 1$ such that for all sufficiently large $X$, the following is true.  The number of elements $\rho$ of $S(X)$ such that $L(s,\rho)$ has a zero in the region
\begin{equation}
\label{eq:nasty}
[1 - \delta, 1] \times [-(\log (2^{40} X^8))^{2/\delta},(\log (2^{40} X^8))^{2/\delta}]
\end{equation} is bounded  by $X^{\tau}$.  
\end{cor}

\begin{proof}  We set $\alpha = 1 - \delta$ and let $T = (\log(2^{40} X^8))^{2/\delta}$.  The constant $c'_0$
in (\ref{eq:c0prime}) is now bounded above by $11$ since $A = 2$, $m \le 2$ and $d = 1$ in Examples \ref{ex:exone} and \ref{ex:extwo}. Choose $c_0 = 12 > c'_0$.  Theorem \ref{thm:bigCheb} shows that there is a bound of the form
\begin{equation}
\label{eq:urk}
\sum_{\rho \in S(X)} N(\rho;\alpha, T) \le M_{4,c_0} T^B X^{c_0(1-\alpha)/(2 \alpha - 1)}
\end{equation}
for some constants $B$ and $M_{4,c_0}$ depending on the parameters associated to the family $(S(X))_{X \ge 1}$
in Examples \ref{ex:exone} and \ref{ex:extwo}.  Since $c_0 = 12 $ is fixed, by taking $\delta \in (0,1/4)$ sufficiently close to $0$, we can ensure that
$$c_0(1-\alpha)/(2 \alpha - 1)  = c_0 \delta/(1 - 2 \delta) < 1.$$
The bound (\ref{eq:urk}) now shows that there is a $\tau < 1$ such that for all sufficiently large $X$ one has
\begin{equation}
\label{eq:aboveit}
\sum_{\rho \in S(X)} N(\rho;\alpha, T) < X^{\tau}.
\end{equation}
This implies that there are at most $X^{\tau}$ elements $\rho \in S(X)$ which could have a zero in the 
region (\ref{eq:nasty}).  
\end{proof}

\begin{cor}
\label{cor:afterlife} With the notations of Corollary \ref{cor:finally}, let $k = \mathbb{Q}$ and suppose $X$ is sufficiently large.  For all but $9 X^{\tau}$ primes $p_i \equiv 1$ mod $8$ for which $p_i \le X$, the field $L_{p_i}$ will satisfy 
\begin{equation}
\label{eq:Chebbound2}
|\pi_{\mathcal{C}}(X,L_{p_i}/\mathbb{Q}) -  \frac{|\mathcal{C}|}{|G|} \mathrm{Li}(X)| \le  \frac{|\mathcal{C}|}{|G|}  \frac{X}{(\log X)^2}
\end{equation}
for all conjugacy classes $\mathcal{C}$ in $G = \mathrm{Gal}(L_{p_i}/\mathbb{Q})$.
\end{cor}

\begin{proof}  We can factor the quotient $\zeta_{L_{p_i}}(s)/\zeta_{\mathbb{Q}}(s)$ into a product of powers of the $L$-functions of the irreducible non-trivial complex representations of $\mathrm{Gal}(L_{p_i}/\mathbb{Q})$.  There are 
$7$ one dimensional representations and $2$ two-dimensional representations of this kind.  Thus Corollary \ref{cor:finally} implies that for all but $9 X^\tau$ primes $p_i \equiv 1$ mod $8$ with $p_i \le X$, there will be no zeros of $\zeta_{L_{p_i}}(s)/\zeta_{\mathbb{Q}}(s)$ in the region (\ref{eq:nasty}).  By Lemma \ref{lem:Frobenius}, one has 
$D_{L_{p_i}} \le 2^{40} p_i^8 \le 2^{40} X^8$, so all but $9 X^\tau$ of the $L_{p_i}$ will satisfy the hypothesis of 
Theorem \ref{thm:uncond}.  The bound $D_{L_{p_i}} \le 2^{40} p_i^8 \le 2^{40} X^8$ also shows that for sufficiently large $X$, (\ref{eq:lowerx1}) will hold when $x = X$.  Hence we have the conclusion of the corollary from Theorem \ref{thm:uncond}.
\end{proof}

\section{End of the proof of Theorem \ref{thm:main}}
\label{s:finish}

We first need to show that there is an  upper bound of the form 
\begin{equation}
\label{eq:Chebbound3}
|\pi_{\mathcal{C}_0}(X,\mathbb{Q}(\zeta_8)/\mathbb{Q}) -  \frac{|\mathcal{C}_0|}{|G|} \mathrm{Li}(X)| \le f(X)
\end{equation}
where $\mathcal{C}_0$ is the conjugacy class of the identity element of $\mathrm{Gal}(\mathbb{Q}(\zeta_8)/\mathbb{Q})$ and $f(X)\ge 0$ is an effectively computable function for which 
$$\lim_{X \to \infty} \frac{|f(X)|}{\mathrm{Li}(X)} = 0.$$
One could use Theorem \ref{thm:uncond} to show this at the cost of having to verify that there is a certain explicit zero free region for $\zeta_{\mathbb{Q}(\zeta_8)}(s)/\zeta_{\mathbb{Q}}(s)$.   However, it is simpler to use the following two results off the shelf.  One can deduce the required bound (\ref{eq:Chebbound3}) by combining Lagarias and Odlyzko's unconditional effective Chebotarev theorem in \cite[Theorem 1.3]{LO} with  Stark's upper bound in
\cite[p. 148]{Stark} on any possible Siegel zero of $\zeta_{\mathbb{Q}(\zeta_8)}(s)$.  For more on this, see \cite[p. 413  - 415]{LO}.

We conclude from (\ref{eq:Chebbound3}) that  the number of primes $p_i \equiv 1$ mod $8$ such that $p_i \le X$ is at least 
$$\frac{1}{4} \mathrm{Li}(X) - f(X).$$
 For each such $p_i$, we apply Lemma \ref{lem:Frobenius} with $\mathcal{C}$ the non-trivial element of $\mathrm{Gal}(L_{p_i}/\mathbb{Q}(\zeta_8,\sqrt{p_i}))$.  This leads to the conclusion that  the number of $p_j \equiv 1$ mod $8$ with $p_i \ne p_j \le X$  and $(p_i,p_j) \in S$ is $\pi_{\mathcal{C}}(X,L_{p_i}/\mathbb{Q})$.  Corollary \ref{cor:afterlife}
now shows that there is a $\tau < 1$ such that for sufficiently large $X$ one has
$$ \pi_{\mathcal{C}}(X,L_{p_i}/\mathbb{Q}) \ge \frac{1}{16} \mathrm{Li}(X) -  \frac{X}{(\log X)^2}$$
for all but $9X^\tau$ of the primes $p_i \equiv 1$ mod $8$ such that $p_i \le X$.  This shows that
for sufficiently large $X$, 
$$\# S \ge \left ( \frac{1}{4} \mathrm{Li}(X) - f(X) - 9 X^\tau  \right )\cdot \left ( \frac{1}{16} \mathrm{Li}(X) -  \frac{X}{(\log X)^2} \right ).$$
This proves that for all $0 < c' < 1/128$ one has 
\begin{equation}
\label{eq:lowerbo}
\#S/2 \ge c' \frac{X^2}{(\log X)^2}
\end{equation}
for all sufficiently large $X$. This inequality and Corollary \ref{cor:niceending} complete the proof.

\section{A connection with R\'edei symbols}
\label{s:redei}

In this section we show a connection between R\'edei symbols and steps 1 through 4 of Construction \ref{con:doit}.  

\begin{lemma}
\label{lem:twoprime} Suppose $L = \mathbb{Q}(\sqrt{p_1},\sqrt{p_2})$ for distinct primes $p_1 \equiv p_2 \equiv 1$ mod $4$.  Let $N(L)$ be the maximal elementary abelian two-extension of $L$ that is unramified outside of infinty.
\begin{enumerate}
\item[(a)] If the quadratic residue symbol $\left ( {p_1} \atop {p_2}\right ) = \left ( {p_2} \atop {p_1}\right ) $ equals $ -1$ then $N(L) = L$.
\item[(b)] Suppose the quadratic residue symbol $\left ( {p_1} \atop {p_2}\right )  = \left ( {p_2} \atop {p_1}\right ) $ equals $1$.  Then $N(L)$ is a quadratic extension of $L$.  The following additional conditions are equivalent in this case:
\vspace{1ex}
\begin{enumerate}
\item[(b1)] The class number of $L$ is odd.
\vspace{1ex}
\item[(b2)] The  quartic residue symbols $\left ( {p_1} \atop {p_2}\right )_4$ and $\left ( {p_2} \atop {p_1}\right )_4 $ satisfy
$\left ( {p_1} \atop {p_2}\right )_4 \cdot \left ( {p_2} \atop {p_1}\right )_4  = -1 $.  
\vspace{1ex}
\item[(b3)] The R\'edei symbol $[p_1,p_2,-1]$ equals $-1$.
\vspace{1ex}
\item[(b4)] $N(L)$ is totally complex.
\end{enumerate}
\vspace{1ex}
Under any of these conditions, $N(L) = L(\sqrt{w})$ for a unit $w \in O_L^*$.   
\end{enumerate}
If $p_1 \equiv p_2 \equiv 1$ mod $8$ and $p_1, p_2 \le X$, then any one of conditions (b1) - (b4) is equivalent to $(p_1,p_2)$ being an
element of the set $S$ of Construction \ref{con:doit}.
\end{lemma} 

\begin{proof} By genus theory, $\mathbb{Q}(\sqrt{p_1})$ has odd narrow class number, and if 
$\left ( {p_1} \atop {p_2}\right ) = -1$ then $L$ is the maximal abelian pro-$2$ extension of $\mathbb{Q}(\sqrt{p_1})$ that 
is unramified outside of infinity and $p_2$ and at most quadratically ramified over the unique prime of $\mathbb{Q}(\sqrt{p_1})$ over $p_2$.  It follows that $N(L) = L$. 

 Suppose now that
$\left ( {p_1} \atop {p_2}\right )  = 1$.  Then by considering the ray class group of $\mathbb{Q}(\sqrt{p_1})$ 
of conductor $(p_2)$ times the infinite places of $\mathbb{Q}(\sqrt{p_1})$ we see that there is a degree four extension of this field that is at most quadratically ramified over $p_2$ and unramified outside of $p_2$ and infinity.  This forces
$N(L)$ to be larger than $L$.  Since the two-rank of the narrow class group of $L$ is bounded above by $1$
by \cite[Theorem 1.1]{KP}, we conclude that this rank is exactly $1$ and $[N(L):L] = 2$.  

The equivalence of (b1) and (b2) is proved by Fr\"ohlich in \cite[Theorem 5.7]{Froh}.  A complex conjugation has order $1$ or $2$ in the (cyclic, non-trivial) Sylow two-subgroup $H$ of the narrow ideal class group of $L$.  Since $\mathrm{Gal}(N(L)/L)$ is the unique order $2$ quotient of $H$, we see that $N(L)$ is totally complex if and only if $H$ has order exactly $2$ and the class number of $L$ is odd.  On the other hand, if the class number of $L$ is odd, then $H$ has exponent $2$, so it must have order exactly two since it is cyclic and non-trivial.  Finally, the R\'edei symbol $[p_1,p_2,-1]$ equals $-1$ if and only if the extension $N(L)/L$ is ramified at infinity \cite[Definition 7.8]{Stevenhagen}.  These arguments imply that (b1) - (b4) are equivalent in case (b).  Finally, since $N(L)$ is quadratic over $L$ in case (b), if $L$ has odd class number then $N(L)= L(\sqrt{w})$ for some unit $w \in O_L^*$.

For the final statement of Lemma \ref{lem:twoprime}, note that if $(p_1,p_2) \in S$ then $L$ has odd class number by Lemma
\ref{lem:algpart1}, so that condition (b1) holds.  Conversely, suppose $p_1, p_2 \le X$ and that (b1) holds.  To show $(p_1,p_2) 
\in S$, it will suffice to show that the condition in step 4 of Construction \ref{con:doit} holds. However, the argument used in Lemma \ref{lem:algpart1}   to show that $(p_j,p_i) \in S$ if $(p_i,p_j) \in S$ shows that if the condition in step 4 fails then $L$ has even class number.  Hence (b1)
implies $(p_1,p_2) \in S$.
\end{proof}

\begin{rem}
\label{rem:stats}   Because of Lemma \ref{lem:twoprime}, our main result on the size of $\#S$ can be reformulated in the following way.
Define functions $a_1, a_2: \mathbb{Z} \to \{0,1\}$ by $a_1(d_1) = a_2(d_1) = 1$ if $d_1$ is a  prime congruent to $1$ mod $8$
and let $a_1(d_1) = a_2(d_1) = 0$ otherwise.  Let $T$ be the set of all pairs $(p_1,p_2)$ of primes congruent to $1$ mod $8$ for which $p_1 \le X$ and $p_2 \le X$.  One sees from \cite[Definition 7.1]{Stevenhagen} that if $(p,p) \in T$ then $[p,p,-1] = 1$.  It follows that
 \begin{eqnarray}
\label{eq:oops}
\sum_{|d_1| < X+1, |d_2| < X+1} a_1(d_1) a_2(d_2) [d_1,d_2,-1]  =  \# T  - 2\cdot \# S \le \left (\frac{1}{16}- 4c' \right ) \left ( \frac{X}{ \mathrm{log}(X)} \right )^2 (1 + o(1))
\end{eqnarray} 
for the positive constant $c'$ appearing in (\ref{eq:lowerbo}), where
$$\# T  =   \frac{1}{16} \left ( \frac{X}{\mathrm{log}(X)}\right  )^2 \cdot (1 + o(1))$$
because of the Chebotarev density theorem.  

We have formulated (\ref{eq:oops}) in a somewhat complicated way in order to compare it to the  trilinear character bound of Koymans and Smith in \cite[Theorem 1.10]{KS}.  Choose three functions $a_{i,j}:\mathbb{Z} \times \mathbb{Z} \to \mathbb{C}$ whose values have absolute value bounded by $1$.  Choose numbers $H_1, H_2, H_3 \ge 3$.  Then Koymans and Smith prove
\begin{eqnarray}
\label{eq:KStorm}
&\displaystyle \left | \ \ \sum_{|d_1| < H_1, |d_2| < H_2, |d_3| < H_3 } a_{1,2}(d_1,d_2)\cdot a_{1, 3}(d_1, d_3) \cdot a_{2,3}(d_2,d_3)\cdot [d_1,d_2,d_3] \ \ \right | & \\
\nonumber
&\displaystyle \ll H_1 H_2 H_3 \cdot \mathrm{log}(H_1 H_2 H_3)^{1792} \cdot \left(H_1^{-1/512} + H_2^{-1/512} + H_3^{-1/512}\right).&
\end{eqnarray}

 The trilinear character bound in (\ref{eq:KStorm}) 
shows that the sum of the R\'edei symbols over certain triples is bounded by the total number of triples times a function that goes to $0$
as long as all of $H_1$, $H_2$ and $H_3$ become large.  It is because of this last requirement that we do not see how to deduce (\ref{eq:oops}) from the trilinear character bound by making $H_3$ small (e.g. by letting $H_3 = 3$).  The right hand side of (\ref{eq:KStorm}) does not then represent a savings over $4 H_1 H_2$ once $H_1$ or $H_2$ are large, and the sum in (\ref{eq:oops}) is over pairs of integers $d_1$ and $d_2$.
In fact, we should not be able to improve the constant $(1/16-4c')$ on the right side of (\ref{eq:oops}) to a function going to $0$ because the Chebotarev results of the previous section show also that there is a positive proportion of pairs of primes congruent to $1$ mod $8$ for which the conditions in Construction \ref{con:doit} fail. 

This raises the question of finding upper bounds on sums of the form
$$\left |\ \  \sum_{|d_1| < H_1, |d_2| < H_2 } A_{1}(d_1) \cdot A_{2}(d_2) \cdot [d_1,d_2,d_3]\ \  \right |$$
when $A_1$ and $A_2$ are arbitrary functions from $\mathbb{Z}$ to  $\mathbb{C}$ whose values are bounded by $1$ and when $d_3$ is a fixed integer.  It would be interesting if a combination of the modular form methods in this paper and the methods of Koymans and Smith in \cite{KS} would give non-trivial bilinear character bounds on such sums.
\end{rem}

 \section{An Example}
 \label{s:example} 
 We will use the notations in (\ref{eq:nudef}) and (\ref{eq:Cdef}).  The numbers $\nu_{i,j}$ lie in 
 $F = \mathbb{Q}(\sqrt{7},\sqrt{13}) $ and are the four conjugates over $\mathbb{Q}$ of 
 $$\nu_{1,1} = -13 - 5 \sqrt{7} + 8 \sqrt{13} + 3 \sqrt{7} \cdot \sqrt{13}.$$
 The norm of $\nu_{1,1}$ to $\mathbb{Q}$ is $-3$.  So $\nu_{1,1}$ generates one of the four split primes over $3$ in
 $F$.  We defined 
 $N'$ to be the field obtained by adjoining to $\mathbb{Q}$ all square roots of elements of  
  $$C = \left\{7,13,-1, \frac{3 + \sqrt{13}}{2}\right\} \cup \{\nu_{i,j}\}_{1 \le i, j \le 2}.$$
  Here $\frac{3 + \sqrt{13}}{2}$ is a fundamental unit of $\mathbb{Q}(\sqrt{13})$, and $\pm \frac{3 + \sqrt{13}}{2}$ and $-1$ are not  squares in $F$.  Since the $\nu_{i,j}$ generate ideals of the integers $O_F$ of $F$ that are independent modulo squares of ideals, we conclude that the elements of $C - \{ 7, 13\}$ generate a subgroup of $F^*/(F^*)^2$ isomorphic to $(\mathbb{Z}/2)^6$.  So by Kummer theory, $N'$ is Galois over $\mathbb{Q}$, $N'$ contains $F$ and $\mathrm{Gal}(N'/F)$ is isomorphic to $(\mathbb{Z}/2)^6$.  Thus $[N':\mathbb{Q}] = 256$.  Since $\frac{3  + \sqrt{13}}{2}$ is a unit, $N'/F$ is unramified outside $2$, $3$ and infinity, and $N'/F$ is tamely and quadratically ramified at each of the four primes over $3$ in $F$. 
 
 There is a unique prime $\mathfrak{P}$ over $2$ in $F$, and $\mathfrak{P}$ is quadratically ramified with residue field of order $4$.  We have $\mathfrak{P} = (3 + \sqrt{7}) O_F$.  We can write
$$1 - \nu_{1,1}  =  14 +  \sqrt{7} \cdot 2 \cdot \left(\frac{5 - 3 \sqrt{13}}{2} \right) - 8 \sqrt{13}.$$ 
This proves $1 - \nu_{1,1} \equiv 0 $ mod $2O_F$, where $2O_F = \mathfrak{P}^2$.  Thus
$$\beta = \frac{1 - \nu_{1,1}}{(3 + \sqrt{7})^2}  \in O_F.$$
Consider the element 
$$\alpha = \frac{1  + \sqrt{\nu_{1,1}}}{3 + \sqrt{7}}$$
of $L = F(\sqrt{\nu_{1,1}})$.  One has
$$\mathrm{Trace}_{L/F}(\alpha) = \frac{2}{3 + \sqrt{7}} \in O_F \quad \mathrm{and}\quad  \mathrm{Norm}_{L/F}(\alpha) = \frac{1 - \nu_{1,1}}{(3 + \sqrt{7})^2}  = \beta \in O_F.$$
Thus $\alpha$ lies in $O_L$.  The discriminant of the minimal polynomial of $\alpha$ over $F$ is 
$$(\alpha - \alpha')^2 = 4 \frac{\nu_{1,1}}{(3 + \sqrt{7})^2}$$
when $\alpha' = (1  - \sqrt{\nu_{1,1}})/(3 + \sqrt{7})$ is the conjugate of $\alpha$ over $F$. Since
$4 O_F = \mathfrak{P}^4$, $(3 + \sqrt{7})O_F = \mathfrak{P}$ and $\nu_{1,1} O_F$ is a first degree prime over $3$, this shows that the relative discriminant ideal $d_{L/F}$ divides $3 \mathfrak{P}^2  = 6O_F$. Thus
$L$ lies in the narrow ray class field of $F$ associated to the ideal $6O_F$, and the same is true for
all of the fields $F(\sqrt{\nu_{i,j}})$ associated to $i,j \in \{1, 2\}$.  

The extension $F(\sqrt{-1})/F$ is unramified outside of infinity and  the unique place $\mathfrak{P}$ over $2$.  The completion of $F$ at $\mathfrak{P}$ contains $\sqrt{-1}$ since $-1 \equiv 7$ mod 8. So $F(\sqrt{-1})/F$
is unramified over all finite places of $F$.

Let $\gamma = \frac{3 + \sqrt{13}}{2}$.  
The extension  $F( \sqrt{\gamma})$ is unramified over all finite places of $F$ apart from possibly $\mathfrak{P}$. One has
$$1 - \gamma^3 = 1 - \frac{27 + 3 \cdot 3^2 \sqrt{13} + 3 \cdot 3 \cdot 13 + 13 \cdot \sqrt{13}}{8} = 2 \cdot \frac{-17 - 5 \sqrt{13}}{2}. $$
Thus $\gamma^3 \equiv 1$ mod $2O_F = \mathfrak{P}^2$.  Now 
$$F(\sqrt{\gamma}) = F(\sqrt{\gamma^3}) = F(\tau)$$
where $\tau = (1 + \sqrt{\gamma^3})/(3 + \sqrt{7})$ lies in $O_F$.  We find that the relative discriminant of 
$F(\sqrt{\gamma})$ over $F$ divides the $O_F$ ideal generated by the discriminant
$$(\tau - \tau')^2 = 4 \frac{\gamma^3}{(3 + \sqrt{7})^2}$$
of $\tau$ over $F$, where $\tau'$ is the conjugate of $\tau$ over $F$.  Here $\gamma$ is a unit, so
the relative discriminant of $F(\sqrt{\gamma})$ over $F$ divides $\mathfrak{P}^2 = 2 O_F$. 

We conclude from these computations that the field $N'$ lies inside the narrow ray class field of $F$ having
conductor $6O_F$ times the infinite places of $F$.  

The fundamental units of the three quadratic subfields of $F$ are 
$\epsilon_7 = 8 - 3 \sqrt{7}$, $\epsilon_{13} = (3 + \sqrt{13})/2$ and  $\epsilon_{91} = 1574 + 165 \sqrt{7} \cdot \sqrt{13}$.  The prime $13$ is inert in $\mathbb{Q}(\sqrt{7})$, and we find by genus theory that $h_F$ is odd.  Here $h_{\mathbb{Q}(\sqrt{7})} = h_{\mathbb{Q}(\sqrt{13})} = 1$ and $h_{\mathbb{Q}(\sqrt{91})} = 2$.
We now find using (\ref{eq:biquadclass}) that $h_F = 1$ and 
$[O_F^*:O_{\mathbb{Q}(\sqrt{7})}^* O_{\mathbb{Q}(\sqrt{13})}^* O_{\mathbb{Q}(\sqrt{91})}^*] = 2$.  One checks that a square root of $\epsilon_{91} \epsilon_{7}$ is given by 
$$\nu = \frac{105 - 33 \sqrt{13} - 45 \sqrt{7} + 11 \sqrt{91}}{2}.$$
Thus $\{\epsilon_7, \epsilon_{13}, \nu\}$ is a set of fundamental units for $F$ and $O_F^*$ is generated by these units together with $-1$.

Let $\mathrm{Cl}_{\mathcal{M}\infty}(F)$ be the narrow ray class group of conductor the product of an integral ideal $\mathcal{M}$ of $O_F$ times the infinite places of $F$.  By calculating the signs of the fundamental units at infinity and their residue classes in $(O_F/\mathcal{M})^*$, one now checks that $\mathrm{Cl}_{3 O_F \infty}(F)$ is isomorphic to $(\mathbb{Z}/2)^4$, while the Sylow two-subgroup of $\mathrm{Cl}_{6O_F \infty}(F)$ is isomorphic to $(\mathbb{Z}/2)^6$.  
Since $N'$ has degree $2^6$ over $F$, it follows that $N'$ is the narrow ray class field of $F$ of conductor $6O_F \infty$, and $N'$ is of degree $4$ over the narrow ray class field $N''$ of $F$ of conductor $3O_F \infty$.  

For each rational prime $p$, let $v(p)$ be a place of $N'$ over $p$.  Define $d_{N'_{v(p)}}$ to be the ideal of $\mathbb{Z}_p$ that is the discriminant of $N'_{v(p)}$ over $\mathbb{Q}_p$.  Let
\begin{equation}
\label{eq:apdef}
a_p = \frac{\mathrm{ord}_p(d_{N'_{v(p)}})}{[N'_{v(p)}:\mathbb{Q}_p]}.
\end{equation}
Since $N'$ is Galois over $\mathbb{Q}$, decomposing the different of $N'$ into a product of local differents shows that 
\begin{equation}
\label{eq:product}
D_{N'}^{1/[N':\mathbb{Q}]} = \prod_p p^{a_p}
\end{equation}
where the product is over all rational primes $p$.

When $p \in \{3, 7, 13\}$, $N'_{v(p)}$ is a Galois extension of $\mathbb{Q}_p$ that is tamely ramified of ramification degree $2$. It follows that 
$$a_p = 1/2 \quad \mathrm{for}\quad p \in \{3, 7, 13\}.$$
The completion $F_{\mathfrak{P}}$ of $F$ at the unique place $\mathfrak{P}$ over $2$
is  the  biquadratic extension $\mathbb{Q}_2(\sqrt{7},\sqrt{13})$ of $\mathbb{Q}_2$, which has ramification and inertia degree equal to $2$.  Let $v''(2)$ be the place of the field $N''$ under $v(2)$.  Then $N''_{v''(2)}/F_{\mathfrak{P}}$ is unramified because $N''$ is the narrow ray classfield of $F$ of conductor $3O_F \infty$.  The fact that $\mathrm{Gal}(N'/N'')$ is a Klein four group shows that $N'_{v(2)}$ is a Klein four extension of
$N''_{v''(2)}$ that has conductor $\mathfrak{P}^2 O_{v''(2)}$ when $O_{v''(2)}$ is the valuation ring of $N''_{v''(2)}$.  The conductor of every non-trivial character of $\mathrm{Gal}(N'_{v(2)}/N''_{v''(2)})$ is $\mathfrak{P}^2 O_{v''(2)}= 2 O_{v''(2)}$.  So we find from the conductor discriminant formula that the relative discriminant 
$d_{N'_{v(2)}/N''_{v''(2)}}$ is 
$2^3 O_{v''(2)}$.  We have $d_{N''_{v''(2)}/F_{\mathfrak{P}}} = O_{\mathfrak{P}}$ since $N''_{v''(2)}/F_{\mathfrak{P}}$ is unramified. Thus
$$d_{N'_{v(2)}/F_{\mathfrak{P}}} = d_{N''_{v''(2)}/F_{\mathfrak{P}}}^{[N'_{v(2)}:N''_{v''(2)}]} \cdot 
\mathrm{Norm}_{N''_{v'(2)}/F_{\mathfrak{P}}} ( d_{N'_{v(2)}/N''_{v''(2)}}) = 2^{3f} O_{\mathfrak{P}}
$$
when $f = [N''_{v''(2)}:F_{\mathfrak{P}}]$.  
Finally, 
$d_{F_\mathfrak{P}} = 4^2 \mathbb{Z}_2$
since $d_{\mathbb{Q}_2(\sqrt{7})}= 4 \mathbb{Z}_2$ and $F_{\mathfrak{P}}/\mathbb{Q}_2(\sqrt{7})$ is unramified.   Thus  
$$d_{N'_{v(2)}/\mathbb{Q}_p} = d_{F_{\mathfrak{P}}}^{[N'_{v(2)}:F_{\mathfrak{P}}]} 
\cdot \mathrm{Norm}_{F_{\mathfrak{P}}/\mathbb{Q}_2}  ( d_{N'_{v(2)}/F_{\mathfrak{P}}} )= d_{F_{\mathfrak{P}}}^{4f} \cdot \mathrm{Norm}_{F_{\mathfrak{P}}/\mathbb{Q}_2} (2^{3f} O_{\mathfrak{P}})  = 2^{16f} \cdot 2^{12f} \mathbb{Z}_2 = 2^{28f} \mathbb{Z}_2$$
so 
$$a_2 = \frac{28f}{16f} = \frac{7}{4}.$$
Since $N'/\mathbb{Q}$ is unramified outside $\{2, 3, 7, 13\}$ and infinity, we find from (\ref{eq:product}) that 
$$
D_{N'}^{1/[N':\mathbb{Q}]} = 2^{7/4} \cdot (3 \cdot 7 \cdot 13)^{1/2} \approx 55.5756...
$$
as in (\ref{eq:Nprimediscriminant}).

\end{document}